\documentclass[11pt]{amsart}

\usepackage{amsfonts}
\usepackage{hyperref}
\usepackage{amsmath,amssymb,amsthm}
\usepackage{enumerate,graphicx}
\usepackage[square,numbers]{natbib}

\newtheorem{theorem}{Theorem}[section]

\newtheorem{proposition}[theorem]{Proposition}
\newtheorem{corollary}[theorem]{Corollary}
\newtheorem{lemma}[theorem]{Lemma}
\theoremstyle{definition}

\newtheorem{definition}[theorem]{Definition}
\newtheorem{example}[theorem]{Example}

\newtheorem{problem}[theorem]{Problem}
\newtheorem{remark}[theorem]{Remark}

\begin{document}
\title[Entropy and SFT on free semigroups]{On the topological entropy of subshifts of finite type on free semigroups}


\author{Jung-Chao Ban}
\address[Jung-Chao Ban]{Department of Applied Mathematics, National Dong Hwa
University, Hualien 970003, Taiwan, R.O.C.}
\email{jcban@gms.ndhu.edu.tw}
\author[Chih-Hung Chang]{Chih-Hung Chang*}
\thanks{*To whom correspondence should be addressed}
\address[Chih-Hung Chang]{Department of Applied Mathematics, National
University of Kaohsiung, Kaohsiung 81148, Taiwan, R.O.C.}
\email{chchang@go.nuk.edu.tw}
\date{January 1, 2018}

\begin{abstract}
In this paper, we provide an effective method to compute the topological
entropies of $G$-subshifts of finite type ($G$-SFTs) with $G=F_{d}$ and $%
S_{d}$, the free group and free semigroup with $d$ generators respectively.
We develop the entropy formula by analyzing the corresponding systems of
nonlinear recursive equations (SNREs). Four types of SNREs of $S_{2}$-SFTs,
namely the types $\mathbf{E},\mathbf{D},\mathbf{C}$ and $\mathbf{O}$\textbf{,%
} are introduced and we could compute their entropies explicitly. This
enables us to give the complete characterization of $S_{2}$-SFTs on two
symbols. That is, the set of entropies of $S_{2}$-SFTs on two symbols is
equal to $\mathbf{E}\cup \mathbf{D}\cup \mathbf{C}\cup \mathbf{O}$. The
methods developed in $S_{d}$-SFTs will also be applied to the study of the
entropy theory of $F_{d}$-SFTs. The entropy formulae of $S_{d}$-, $F_{d}$%
-golden mean shifts and $k$-colored chessboards are also presented herein.
\end{abstract}
\maketitle

% Title Page ----------------------------------------------------
\baselineskip=1.3 \baselineskip

\section{Introduction}

A classical dynamical system is a pair $(X,T)$ in which $X$ is a phase space
and $T:\mathbb{R}^{+}\times X\rightarrow X$ is a function that describes the
evolution of elements of $X$. If a dynamical system is hyperbolic, the
symbolic dynamical system is an essential and useful tool to investigate the
original system \cite{Bowen-1975, Ruelle-1978}. That is, there is a
partition $\mathcal{P}=\{P_{1},\ldots ,P_{m}\}$ of $X$ and a \emph{coding
map }$\pi :X\rightarrow \{1,\ldots ,m\}^{\mathbb{N}\cup \{0\}}$ such that $%
\sigma \circ \pi =\pi \circ T$, where $\sigma :(\xi _{n})_{n=0}^{\infty
}\longmapsto (\xi _{n+1})_{n=0}^{\infty }$ is the \emph{shift map} on $\pi
(X)$. One could gain almost all information of $(X,T)$ from $(\pi (X),\sigma
)$. The system $(\pi (X),\sigma )$ is also called a $\mathbb{Z}^{1}$\emph{%
-symbolic dynamical system}. A natural extension of $\mathbb{Z}^{1}$%
-symbolic dynamical system is the $G$\emph{-dynamical system} which is a
pair $(X,T)$ such that $X$ is a topological space and $T:G\times
X\rightarrow X$ is a group action on $X$. These spaces appear naturally as
discrete versions of dynamical systems. A pair $(X,\sigma )$ is called a $%
\emph{G}$\emph{-}$\emph{subshift}$ it there exists a finite set $\mathcal{A}$
such that $X$ is a closed subset of $\mathcal{A}^{G}$ and $\sigma :G\times
X\rightarrow X$ is defined by $\left( \sigma ^{g}x\right) (h)=x(g^{-1}h)$. A 
$G$-dynamical system is topologically conjugate to a $G$-subshift if and
only if it is zero-dimensional and expansive ($G=\mathbb{Z}^{1}$ \cite%
{hedlund1969endomorphisms}, general group \cite{lemp2017shift}), it is
therefore of interest to study the $G$-subshifts. The aim of this paper is
to investigate the case of $F_{d}$- and $S_{d}$-subshifts, where $F_{d}$ and 
$S_{d}$ are the free group and free semigroup, respectively.

Let $F_{d}$ be a free group that is generated by $S=\{s_{1},\ldots ,s_{d}\}$
and $\mathcal{A}$ is a symbol set which is finite. A $\emph{coloring}$ is a
function $f:F_{d}\rightarrow \mathcal{A}$ and a set $\mathcal{F}\subset 
\mathcal{A}\times S\times \mathcal{A}$ is called \emph{forbidden transitions}. A
coloring $f$ \emph{contains }$(a,s_{i},b)\in \mathcal{F}$ if there exists $%
g\in F_{d}$ such that $f(g)=a$ and $f(gs_{i})=b$. Denote by $X_{\mathcal{F}}$
the set of colorings which do not contain any block in $\mathcal{F}$. Since $%
\mathcal{F}$ is finite, we call $(X_{\mathcal{F}},\sigma )$ a $G$\emph{-SFT}%
, where the action $\sigma $ is given by $\left( \sigma _{i}f\right)
(x)=f(s_{i}^{-1}x)$ for $x\in F_{d}$, $1\leq i\leq d$. A \emph{pattern }with
support $F\subset F_{d}$ is an element $p\in \mathcal{A}^{F}$ and we write
supp$(p)=F$. Let $n\in \mathbb{N}$, we denote by $E_{n}$ the set of elements
in $G$ whose length are less than or equal to $n$ and denote by $B_{n}(X_{%
\mathcal{F}})$ the set of all possible colorings of $X_{\mathcal{F}}$ in $%
E_{n}$. The \emph{topological entropy} (\emph{entropy} for short) of $X_{%
\mathcal{F}}$ is defined as 
\begin{equation}
h(X_{\mathcal{F}})=\limsup_{n\rightarrow \infty }\frac{\ln \left\vert B_{n}(X_{%
\mathcal{F}})\right\vert }{\left\vert E_{n}\right\vert },  \label{17}
\end{equation}%
where the limit (\ref{17}) exists for $G=S_d$ due to the recent result of
Petersen-Salama (Theorem 5.1 \cite{PS-2017complexity}). It is an interesting
problem to find the explicit value or formula of (\ref{17}). For the case
where $G=\mathbb{Z}^{d}$ with $d=1$, the entropy formulae of $\mathbb{Z}^{1}$%
-SFTs and its algebraic characterization are given by D. Lind \cite%
{Lind-ETaDS1984, LM-1995}. Precisely, the nonzero entropies of $\mathbb{Z}^{1}$-SFTs
are exactly the non-negative rational multiples of logarithm of Perron
numbers. (A Perron number is a real algebraic integer greater than 1 and
greater than the modulus of its algebraic conjugates.) For $d\geq 2$,
Hochman-Meyerovitch \cite{HM-AoM2010a} indicate that the entropy of a $%
\mathbb{Z}^{d}$-SFT is \emph{right recursively enumerable}, i.e., it is the
infimum of a monotonic recursive sequence of rational numbers. 
%Later, Pavlov-Schraudner \cite{PS-TotAMS2015} show that, for every $d\geq 3$ and every $\mathbb{Z}^{d}$-full shift, there is a block gluing $\mathbb{Z}^{d}$-SFT which shares identical entropy. 
However, what is still lacking is the
general entropy formulae of $\mathbb{Z}^{d}$-SFTs for $d\geq 2$. Some
approximation algorithms for the entropies of $\mathbb{Z}^{d}$-SFTs can be
found in \cite{friedland-2003, MP-ETaDS2013, MP-SJoDM2013}.

Our goal is to find the entropy formulae of $S_{d}$- and $F_{d}$-SFTs. For
an $S_{d}$-SFT $X_{\mathcal{F}}$, the number $\left\vert B_{n}(X_{\mathcal{F}%
})\right\vert $ satisfies the so-called \emph{system of nonlinear recursive
equation }\cite{ban2017tree} (SNRE, defined in Section 2), thus the
computation of (\ref{17}) is to solve its corresponding SNRE. The main
difficulty in carrying out this approach is that $\left\vert B_{n}(X_{%
\mathcal{F}})\right\vert $ behaves approximately like $\lambda _{1}\lambda
_{2}^{\kappa ^{n}}$ for some $\lambda _{1},\lambda _{2}$ and $\kappa \in 
\mathbb{R}$. That is,%
\begin{equation}
\left\vert B_{n}(X_{\mathcal{F}})\right\vert \approx \lambda _{1}\lambda
_{2}^{\kappa ^{n}},  \label{18}
\end{equation}%
the number $\ln \kappa $ is called the \emph{degree }of $X_{\mathcal{F}}$
and one can derive it via the following formula $\lim_{n\rightarrow \infty }%
\frac{\ln ^{2}\left\vert B_{n}(X_{\mathcal{F}})\right\vert }{n}$, where $\ln
^{2}=\ln \circ \ln $. If $\kappa =d$, we easily obtain that $h(X_{\mathcal{F}%
})=\ln \lambda _{2}$. (\ref{18}) can be interpreted in the following two
perspectives. \textbf{1. (symbols and generators)}: If $d=1$, we consider
the one-sided $\mathbb{Z}^{1}$-subshifts, it is known that $\left\vert
B_{n}(X_{\mathcal{F}})\right\vert \approx \lambda _{1}\lambda _{2}^{n}$ and
one can regard such formulation as that we could use $\lambda _{2}$ (in
average) colors to \emph{fill up }the elements of $E_{n}$ in $\mathbb{Z}^{1}$%
, e.g., if $X$ is a full shift with $\mathcal{A}=\{1,2\}$, then $\left\vert
B_{n}(X_{\mathcal{F}})\right\vert =2^{n}$ while $\left\vert E_{n}\right\vert
=n$. Similarly, if $G=S_{d}$, the formula (\ref{18}) represents that we
could use $\lambda _{2}$ colors to fill up the elements of $E_{n}$ in $%
S_{\kappa }$ (note that $\left\vert E_{n}\right\vert \approx d^{n}$ in $S_{d}
$). Thus, it can be symbolized as $\left\vert B_{n}(X_{\mathcal{F}%
})\right\vert \approx \lambda _{1}\lambda _{2}^{\kappa ^{n}}\approx
\left\vert \mathcal{A}\right\vert ^{\left\vert E_{n}\right\vert }$. The
important point to note here is that the value $\kappa $ is not always an
integer. In fact, we show that $\{\ln \kappa :\ln \kappa $ is a degree of an 
$S_{d}$-SFT$\}=\{\frac{1}{p}\ln \lambda :\lambda \in \mathcal{P},$ $p\geq 1\}
$, where $\mathcal{P}$ stands for the set of Perron numbers \cite%
{ban2017characterization}; \textbf{2. (speed)}: It can be easily seen that
the value $\kappa $ (resp. $\lambda _{2}$ and $\lambda _{1}$) indicates the
highest (resp. the middle and the lowest) speed of $\left\vert B_{n}(X_{%
\mathcal{F}})\right\vert $ in (\ref{18}). The more detailed information of
the tuple $(\kappa ,\lambda _{2},\lambda _{1})\in \mathbb{R}^{3}$ we know,
the more explicit value of $\left\vert B_{n}(X_{\mathcal{F}})\right\vert $
we obtain.

We will restrict our attention to the entropy formulae of $S_{d}$-SFTs. Due
to the equivalence of the studies of $\left\vert B_{n}(X_{\mathcal{F}%
})\right\vert $ and SNRE, we embark our study on analyzing all kinds of
SNREs. Various types of SNREs, namely, the\emph{\ equal growth}, \emph{%
dominating}, \emph{oscillating} and \emph{cooperate types} (write $\mathbf{E}
$, $\mathbf{D}$, $\mathbf{O}$, $\mathbf{C}$ respectively) are presented in
Section 3, and the entropy formulae of these types are presented therein.
Furthermore, we give the complete characterization of $h(X_{\mathcal{F}})$
of $(d,k)=(2,2)$. That is, all the SNREs of $S_{2}$-SFTs with $\mathcal{A}%
=\{1,2\}$ are equal to $\mathbf{E}\cup \mathbf{D}\cup \mathbf{O}\cup \mathbf{%
C}$.

The \emph{hom-shifts}, which is an important class of $G$-SFTs, are defined
in \cite{chandgotia2016mixing} and motivated from the mathematical and
statistical physics. A hom-shift can be described as a nearest neighborhood
SFT with the `symmetric' and `isotropic' properties. That is, if $%
(a,s_{i},b)\in \mathcal{F}$ for some $1\leq i\leq d$ then $(a,s_{j},b)\in 
\mathcal{F}$ for all $1\leq j\leq d$. If $\mathcal{A}=\{1,2\}$ and $G=%
\mathbb{Z}^{2}$, $X_{\mathcal{F}}$ is the well-known \emph{two-dimensional
golden mean shift }$(\mathbb{Z}^{2}$-$GMS)$ if 
\begin{equation}
\mathcal{F}=\{(2,s_{i},2):i=1,2\}\text{.}  \label{19}
\end{equation}%
It is obvious that a $\mathbb{Z}^{2}$\emph{-GMS }is a hom-shift and its
entropy $h(X_{\mathcal{F}})$ is called the $\emph{hard}$ $\emph{square}$ $%
\emph{entropy}$ $\emph{constant}$. Either finding a closed form for $h(X_{%
\mathcal{F}})$ or to determine whether this number is algebraic is still an
open problem. Two important types of hom-shifts on $S_{d}$, the $S_{d}$-GMS
and the $k$-colored\emph{\ chessboard} on $S_{d}$, are introduced and their
entropy formulae are established in Section 4. Recently, Petersen-Salama 
\cite{PS-2017complexity} find upper and lower bounds of a hom-shift by means
of the topological entropies of one-dimensional SFTs, and such result gives
the relation between $S_{d}$-SFTs and $\mathbb{Z}^{1}$-SFTs. If $X_{\mathcal{%
F}}$ is an $F_{d}$-GMS on two symbols, i.e., $\mathcal{F}=(\ref{19})$,
Piantadosi \cite{piantadosi2008symbolic} studies the tuple $(\kappa ,\lambda
_{2},\lambda _{1})\in \mathbb{R}^{3}$ of (\ref{18}) and shows that $\kappa
=2d-1$, $\lambda _{2}\approx 0.909155$. Meanwhile, the value $\lambda _{1}$
is not well-defined. It is unclear that whether there is some lower-order indicator between $\lambda_1$ and $\lambda_2$. The results of $S_{d}$-SFTs developed in Section 3 can
be applied to find the entropy of $F_{d}$-GMS, which provides a different
approach than \cite{piantadosi2008symbolic}. Some open problems on the
entropy theory of $F_{d}$- or $S_{d}$-SFTs are illustrated in Section 5.

\section{Preliminaries}

In this section we set up the notations of $S_{d}$-SFTs and present some
known results. Some notations will be adjusted to meet the conditions of
previous work in \cite{ban2017tree}. Let $\mathcal{A}=\{1,\ldots ,k\}$ be
the symbol set. Suppose $S_{d}$ is a free semigroup with $d$ generators,
then it is also a $d$-tree. Thus, the shift space $(\mathcal{A}%
^{S_{d}},\sigma )$ is also called a \emph{full} \emph{tree-shift} in \cite%
{AB-TCS2012, ban2017tree}. Let $\Sigma =\{1,\ldots ,d\}$ and $\Sigma ^{\ast
}=\cup _{n\geq 0}\Sigma ^{n}$ be the union of finite words over $\Sigma $,
where $\Sigma ^{n}=\{w_{1}w_{2}\cdots w_{n}:w_{i}\in \Sigma $ for $1\leq
i\leq n\}$ is the collection of words of length $n$ for $n\in \mathbb{N}$
and $\Sigma ^{0}=\{\epsilon \}$ consists of the empty word $\epsilon $. A 
\emph{coloring} (also call a \emph{labeled tree}) in $\mathcal{A}^{S_{d}}$
is a function $f:\Sigma ^{\ast }\rightarrow \mathcal{A}$. For $%
w=w_{0}w_{1}\ldots w_{n-1}\in \Sigma ^{\ast }$, define $\sigma _{w}=\sigma
_{n-1}\cdots \sigma _{1}\sigma _{0}$ and $\left( \sigma _{i}f\right)
(w)=f(iw)$ for $i=1,\ldots ,d$. Suppose $n\in \mathbb{N}\cup \{0\}$, $\Sigma
_{n}=\cup _{k=0}^{n}\Sigma ^{k}$ denotes the set of words of length at most $%
n$, we call a function $u:\Sigma _{n}\rightarrow \mathcal{A}$ is an $n$\emph{%
-block }and write height$(u)=n$. The $n$-block $u$ can be written as $%
u=(u_{\epsilon },\sigma _{1}u,\ldots ,\sigma _{d}u)$, where $\sigma _{i}u$
is the $\left( n-1\right) $-block for $i=1,\ldots ,d$. Let $\mathcal{B}$ be
the collection of $2$-blocks, we use the notation $u=(i,i_{1},\ldots ,i_{d})$
to denote $u\in \mathcal{B}$. Given the forbidden set $\mathcal{F}$, one can
construct the associated set of \emph{admissibl}$\emph{e}$ $2$-blocks (also
call \emph{basic set}) as $\mathcal{B}=\cup _{i=1}^{k}\mathcal{B}^{(i)}$,
where $\mathcal{B}^{(i)}=\{(i,i_{1},\ldots ,i_{d}):(i,s_{j},i_{j})\notin 
\mathcal{F}$, $1\leq j\leq d\}$. Let $X^{\mathcal{B}}$ $(=X_{\mathcal{F}})$
be the $S_{d}$-SFT constructed by $\mathcal{B}$.

\begin{example}
\label{Ex: 1}Let $d=2$, $\mathcal{A}=\{1,2\}$ and $\mathcal{F}%
=\{(2,s_{i},2):i=1,2\}$. The corresponding basic set is $\mathcal{B}=\cup
_{i=1}^{2}\mathcal{B}^{(i)}$, where 
\begin{equation*}
\mathcal{B}^{(1)}=\{(1,1,1),(1,1,2),(1,2,1),(1,2,2)\}\text{ and }\mathcal{B}%
^{(2)}=\{(2,1,1)\}\text{.}
\end{equation*}%
Then $X^{\mathcal{B}}=X_{\mathcal{F}}$ is the $S_{2}$-GMS.
\end{example}

For an $X^{\mathcal{B}}$, the system of nonlinear recursive equations is a
useful tool to analyze the growth behavior of $\left\vert B_{n}(X^{\mathcal{B%
}})\right\vert $ \cite{ban2017tree}. Denote $\Delta =\{\delta _{1},\delta
_{2},\ldots ,\delta _{k}\}$ and $\Delta _{n}^{d}=\{\delta
_{1;n}^{i_{1}}\cdots \delta _{k;n}^{i_{k}}:i_{1}+\cdots +i_{k}=d\}$ for $%
n\geq 1$.

\begin{definition}
\label{Def: 1}

\begin{enumerate}
\item Let $F=\sum_{\mathbf{a}\in \Delta _{n}^{d}}f_{\mathbf{a}}\mathbf{a}$,
where $f_{\mathbf{a}}\in \mathbb{Z}^{+}$ and $\mathbf{a}\in \Delta _{n}^{d}$%
. The vector $v_{F}=(f_{\mathbf{a}})_{\mathbf{a}\in \Delta _{n}^{d}}$ is
called the \emph{indicator vector }of $F$.

\item A sequence $\{\delta _{1;n},\ldots ,\delta _{k;n}\}_{n\in \mathbb{N}}$
is defined by a \emph{system of nonlinear recursive equation} (SNRE) of
degree $(d,k)$ if 
\begin{equation*}
\delta _{i;n}=F^{(i)}=F^{(i)}(\delta _{1;n-1},\delta _{2;n-1},\ldots ,\delta
_{k;n-1}),\text{ }n\geq 2,\text{ }1\leq i\leq k,
\end{equation*}%
and $\delta _{i;1}=\delta ^{(i)}$ for $1\leq i\leq k$, where $F^{(1)},\ldots
,F^{(k)}$ are combinations of $\Delta _{n-1}^{d}$ over $\mathbb{Z}^{+}$,
respectively. We also call $F=\{F^{(1)},\ldots ,F^{(k)}\}$ an SNRE.
\end{enumerate}
\end{definition}

For $1\leq i\leq k$ and $n\in \mathbb{N}$, let $\gamma _{i;n}$ be the number
of $n$-blocks of $X^{\mathcal{B}}$ with the root $\epsilon $ is colored by
the symbol $i$. Denote by $\eta _{n}$ a bijection from $\mathcal{A}$ to $%
\{\gamma _{1;n},\ldots ,\gamma _{k;n}\}$ with $\eta _{n}(i)=\gamma _{i;n}$
and observe that $\gamma _{i;0}=1$ for $1\leq i\leq k$ and $n\geq 1$. The
following result indicates that the number $\left\vert B_{n}(X^{\mathcal{B}%
})\right\vert $ can be described by its corresponding SNRE.

\begin{theorem}[Ban-Chang 2017 \protect\cite{ban2017tree}]
\label{Thm: 4}Given a basic set $\mathcal{B}$ and the corresponding $S_{d}$%
-SFT $X^{\mathcal{B}}$. Then the values $\{\gamma _{i;n}\}_{n=1}^{k}$
satisfies the following SNRE.%
\begin{equation}
\left\{ 
\begin{array}{c}
\gamma _{i;n}=\sum\limits_{(i,i_{1},\ldots ,i_{d})\in \mathcal{B}%
}\prod_{j=1}^{d}\eta _{n-1}(i_{j})\text{,} \\ 
\gamma _{i;1}=\sum\limits_{(i,i_{1},i_{2},\ldots ,i_{d})\in \mathcal{B}%
}\prod_{j=1}^{d}\eta _{0}(i_{j}) \\ 
=\left\vert B_{1}(X_{i}^{\mathcal{B}})\right\vert \text{, }1\leq i\leq k%
\text{,}%
\end{array}%
\right.  \label{20}
\end{equation}%
where $X_{i}^{\mathcal{B}}=\{f\in \mathcal{A}^{S_{d}}:f(\epsilon )=i\}$.
\end{theorem}

Instead of proving Theorem \ref{Thm: 4} we give an example below so that our
exposition is self-contained.

\begin{example}[Continued]
\label{Ex: 2}Let $X^{\mathcal{B}}$ be defined as in Example \ref{Ex: 1}.
Since $\mathcal{B}^{(1)}=\{(1,1,1),(1,1,2),(1,2,1),(1,2,2)\}$, we obtain a
recursive formula for $\gamma _{1;n}$ by (\ref{20}) as follows. 
\begin{equation*}
\gamma _{1;n}=\left( \gamma _{1;n-1}\right) ^{2}+2\gamma _{1;n-1}\gamma
_{2;n-1}+\left( \gamma _{2;n-1}\right) ^{2}\text{.}
\end{equation*}%
Similarly, since $\mathcal{B}^{(2)}=\{(2,1,1)\}$, we have $\gamma
_{2;n}=\left( \gamma _{1;n-1}\right) ^{2}$. The SNRE of $X^{\mathcal{B}}$ is
as follows. 
\begin{equation*}
\left\{ 
\begin{array}{c}
\gamma _{1;n}=\left( \gamma _{1;n-1}\right) ^{2}+2\gamma _{1;n-1}\gamma
_{2;n-1}+\left( \gamma _{2;n-1}\right) ^{2}\text{,} \\ 
\gamma _{2;n}=\left( \gamma _{1;n-1}\right) ^{2}\text{,} \\ 
\gamma _{1;1}=4\text{ and }\gamma _{2;1}=1\text{.}%
\end{array}%
\right.
\end{equation*}
\end{example}

Let $\Gamma =\{\gamma _{1},\ldots ,\gamma _{k}\}$ and $\Gamma _{n}=\{\gamma
_{1;n},\ldots ,\gamma _{k;n}\}$ for $n\geq 1$. Suppose $F$ is an SNRE. For $%
a,b\in \Gamma $, we say that $a$\emph{\ induces }$b$ (or $b$\emph{\ appears
in }$a$), write $a\rightarrow b$, if $b_{n-1}$ appears in some term of $a_{n}
$. Precisely, $a_{n}=f_{n-1}b_{n-1}+$ $g_{n-1}$, where $f_{n-1}$ is a
combination of symbols $\Gamma _{n-1}\backslash \{b_{n-1}\}$. We say that $a$
\emph{connects to} $b$ if there exists a path $a=c_{0},c_{1},c_{2},\ldots
,c_{m}=b$ such that $c_{i}\rightarrow c_{i+1}$ for $i=0,\ldots ,m-1$. Thus
we can decompose the symbols $\Gamma =\Gamma ^{1}\cup \Gamma ^{2}\cup \cdots
\cup \Gamma ^{r}$ as for each $a,b\in \Gamma ^{i}$, either $a$ connects to $b
$ or $b$ connects to $a$. Let $a\in \Gamma $. We say that $a$ is an \emph{%
essential symbol }if there exists an integer $n\in \mathbb{N}$ such that $%
a_{n}\geq 2$ and call $a$ \emph{inessential }if it is not essential, i.e., $%
a_{n}=1$ for all $n\in \mathbb{N}$. We use the notation $\mathcal{E}(F)$
(resp. $\mathcal{I}(F)$) to denote the set of the essential symbols (resp.
inessential symbols) of $F$.

\begin{lemma}
\label{Lma: 3}Let $F$ be an SNRE. If $b\in \mathcal{E}(F)$ and $a\rightarrow
b$, then $a\in \mathcal{E}(F)$.
\end{lemma}

\begin{proof}
Since $b\in \mathcal{E}(F)$, there exists $k$ such that $b_{k}\geq 2$.
Furthermore, $a\rightarrow b$ means that $a_{n}=f_{n-1}b_{n-1}+g_{n-1}$.
Thus $a_{k}=f_{k-1}b_{k-1}+g_{k-1}\geq 2$. This completes the proof.
\end{proof}

\begin{theorem}
Given an SNRE $F$, there exists an algorithm to find $\mathcal{E}(F)$ and $%
\mathcal{I}(F)$. Such an algorithm halts up to $k$-steps.
\end{theorem}

\begin{proof}
Without loss of generality, we assume $r=1$. Let $\mathcal{S}_{1}\subseteq
\Gamma $ be the set of $a\in \mathcal{S}_{1}$ with $a_{1}\geq 2$. We define $%
\mathcal{S}_{k}$ by induction as for any $a\in \mathcal{S}_{k}$ there exists 
$b\in \mathcal{S}_{k-1}$ such that $a\rightarrow b$. Since $\Gamma $ is
finite ($\left\vert \Gamma \right\vert =k$), this algorithm will stop up to $%
k$-steps. Thus we have $\mathcal{S}=\mathcal{S}_{1}\cup \mathcal{S}_{2}\cup
\cdots \cup \mathcal{S}_{q}\subseteq \Gamma $. It follows from Lemma \ref%
{Lma: 3} that $\mathcal{S}\subseteq \mathcal{E}(F).$ Let $\mathcal{R}=\Gamma
\backslash \mathcal{S}$, we claim that $\mathcal{R}=\mathcal{I}(F)$. If $%
a\in \mathcal{R}$, this means that $a\notin \mathcal{S}_{k}$ for $k=1,\ldots
,q$, thus $a_{1}=1$. Suppose there is an integer $m\in \mathbb{N}$ such that 
$a_{m}\geq 2$. Since $r=1$, this means that $a$ must induce some $b\in 
\mathcal{S}$ or $a_{m}$ has at least two items. If $a_{m}$ has at least two
terms, then $a_{1}\geq 2$, which is a contradiction. If $a$ induces some $%
b\in \mathcal{S}$, then $a\in \mathcal{S}$, which is also a contradiction.
Thus $a_{n}=1$ for all $n\in \mathbb{N}$. That is, $a\in \mathcal{I}(F)$. On
the other hand, if $a\in \mathcal{I}(F)$, it follows from Lemma \ref{Lma: 3}
that $a$ cannot induce any $b\in \mathcal{E}(F)$. This means that $a\notin 
\mathcal{S}$, i.e., $a\in \mathcal{R}$. Thus $\mathcal{R}=\mathcal{I}(F)$.
This completes the proof.
\end{proof}

\section{Topological entropy}

In this section, we introduce four types of SNREs that we can compute their
entropies explicitly. For simplicity of discussion and notation, we assume $%
(d,k)=(2,2)$ throughout this section. In this circumstance, we define $a_{n}=\gamma
_{1;n}$ and $b_{n}=\gamma _{2;n}$ and use $F=\{F^{(a)},F^{(b)}\}$ to denote
the associated SNRE. For example, in Example \ref{Ex: 2} 
\begin{equation*}
\left\{ 
\begin{array}{c}
a_{n}=F^{(a)}=a_{n-1}^{2}+2a_{n-1}b_{n-1}+b_{n-1}^{2}, \\ 
b_{n}=F^{(b)}=a_{n-1}^{2}.%
\end{array}%
\right.
\end{equation*}

\begin{remark}
\label{Rmk: 1}The notation $F=\{F^{(a)},F^{(b)}\}$ is adequate to describe
the SNRE. Let $\left\vert F^{(\ast )}\right\vert $ be the number of items of 
$F^{(\ast )}$ for $\ast =a,b$. In the Definition \ref{Def: 1} of SNRE, we
still need the initial conditions, i.e., $a_{1}$ and $b_{1}$, to define an
SNRE. However, we see that $\left\vert F^{(a)}\right\vert =a_{1}$ and $%
\left\vert F^{(b)}\right\vert =b_{1}$ in (\ref{20}). Thus, $%
F=\{F^{(a)},F^{(b)}\}$ provides all information of an SNRE.
\end{remark}

\subsection{Equal growth type}

We say an SNRE is of the \emph{equal growth type}, write type $\mathbf{E}$,
if $a_{n}=b_{n}$ for all $n\geq 1$.

\begin{proposition}
\label{Prop: 1}Let $F$ be an SNRE. If $a_{1}=b_{1}$, then $a_{n}=b_{n}$ for
all $n\in \mathbb{N}$.
\end{proposition}

\begin{proof}
Assume $a_{n-1}=b_{n-1}$. The fact of $a_{1}=b_{1}$ indicates that $F^{(a)}$
and $F^{(b)}$ have the same number of items (see Remark \ref{Rmk: 1}), and
each item is a combination of $a_{n-1}$ and $b_{n-1}$ with degree $d$. Since 
$a_{n-1}=b_{n-1}$, one can compares each item of $F^{(a)}$ with $F^{(b)}$ to
conclude that $a_{n}=b_{n}$. The proof is thus completed by mathematical
induction.
\end{proof}

\begin{proposition}
\label{Prop: 2}Let $F$ be an SNRE with $a_{1}=b_{1}$, then $h(X^{\mathcal{B}%
})=\frac{1}{2}\ln a_{1}$.
\end{proposition}

\begin{proof}
Since $a_{n}=b_{n}$ and $\left\vert E_{n}\right\vert =2^{n+1}-1$. According
to (\ref{17}), we have 
\begin{equation*}
h(X^{\mathcal{B}})=\lim_{n\rightarrow \infty }\frac{1}{2^{n+1}-1}\ln
(a_{n}+b_{n})=\lim_{n\rightarrow \infty }\frac{1}{2^{n+1}-1}\ln a_{n}.
\end{equation*}%
It follows from Proposition \ref{Prop: 1} that we rewrite $a_{n}$ as $%
a_{n}=a_{1}a_{n-1}^{2}$. Let $\alpha _{n}=\ln a_{n}$, we have $\alpha
_{n}=2\alpha _{n-1}+\alpha _{1}$. Iterating $\alpha _{n}$ yields $\alpha
_{n}=2^{n-1}\alpha _{1}\sum_{i=0}^{n-1}2^{-i}=2^{n}\alpha _{1}(1-\left( 
\frac{1}{2}\right) ^{n})$. Thus 
\begin{equation*}
h(X^{\mathcal{B}})=\lim_{n\rightarrow \infty }\frac{2^{n}\alpha
_{1}(1-\left( \frac{1}{2}\right) ^{n})}{2^{n+1}-1}=\frac{1}{2}\alpha _{1}=%
\frac{1}{2}\ln a_{1}.
\end{equation*}%
This completes the proof.
\end{proof}

\subsection{Dominating type}

Let $F$ be an SNRE, $a\in \mathcal{A}$ is called a \emph{dominate symbol }if 
$a_{n}\geq b_{n}$ for $n\in \mathbb{N}$. If $F$ admits a dominate symbol, we
call $F$ is of the \emph{dominating} $\emph{type}$ (type $\mathbf{D}$).

\begin{example}
Suppose $F$ is defined as 
\begin{equation*}
\left\{ 
\begin{array}{c}
a_{n}=a_{n-1}^{2}+2a_{n-1}b_{n-1} \\ 
b_{n}=a_{n-1}^{2}+a_{n-1}b_{n-1} \\ 
a_{1}=3\text{, }b_{1}=2.%
\end{array}%
\right.
\end{equation*}%
Since $a_{1}\geq b_{1}$ and if we assume $a_{n-1}\geq b_{n-1}$, we have $%
a_{n}\geq b_{n}$ according to $F$. Thus $a$ is a dominate symbol and $F\in 
\mathbf{D}$.
\end{example}

\begin{lemma}
\label{Lma: 2}Let $F$ be an SNRE. If $a\in \mathcal{A}$ is a dominate
symbol, then $h(X^{\mathcal{B}})=\lim_{n\rightarrow \infty }\frac{\ln a_{n}}{%
2^{n+1}-1}$.
\end{lemma}

\begin{proof}
If $a$ is a dominate symbol, then $1<1+\frac{b_{n}}{a_{n}}\leq 2$. Then $%
h(X^{\mathcal{B}})=\lim_{n\rightarrow \infty }\frac{\ln a_{n}(1+\frac{b_{n}}{%
a_{n}})}{2^{n+1}-1}=\lim_{n\rightarrow \infty }\frac{\ln a_{n}}{2^{n+1}-1}$.
This completes the proof.
\end{proof}

Let $\mathcal{M}_{m}$ be the collection of $m\times m$ square matrices.
Suppose $F=\{F^{(a)},F^{(b)}\}$ admits a dominate symbol, say $a$, we could
arrange items of $F^{(a)}$ and $F^{(b)}$ with respect to the power of $%
a_{n-1}$ descendingly. Denote by $F_{1}^{(\ast )}$ be the first item of $%
F^{(\ast )}$ for $\ast =a$ or $b$. For $\ast =$ $a$ or $b$, we have 
\begin{equation}
F^{(\ast )}=c_{1}F_{1}^{(\ast )}+c_{2}F_{2}^{(\ast )}\cdots +c_{k_{\ast
}}F_{k_{\ast }}^{(\ast )}=F_{1}^{(\ast )}(c_{1}+\cdots +\frac{c_{k_{\ast
}}F_{k_{\ast }}^{(\ast )}}{F_{1}^{(\ast )}})\text{.}  \label{22}
\end{equation}%
Let 
\begin{equation}
r_{n-1}^{(\ast )}=(c_{1}+\cdots +\frac{c_{k_{\ast }}F_{k_{\ast }}^{(\ast )}}{%
F_{1}^{(\ast )}}).  \label{21}
\end{equation}%
Here we use the notation $r_{n-1}^{(\ast )}$ since the RHS of (\ref{21}) is
a combination of $\{a_{n-1},b_{n-1}\}$. Since $F_{1}^{(\ast )}\geq
F_{j}^{(\ast )}$ for all $j=2,\ldots ,k_{\ast }$, we have 
\begin{equation}
1\leq r^{(\ast )}\leq c_{1}+\cdots +c_{k_{\ast }}\leq 4\text{,}  \label{23}
\end{equation}%
where the number $4$ comes from the extreme case of $c_{1}=1$ and $k_{\ast
}=4$. Let $\alpha _{n}=\ln a_{n}$ and $\beta _{n}=\ln b_{n}$ and $%
v_{n}=\left( 
\begin{array}{c}
\alpha _{n} \\ 
\beta _{n}%
\end{array}%
\right) $. Combining these with (\ref{22}), we deduce 
\begin{equation}
v_{n}=Kv_{n-1}+\ln r_{n-1}\text{, where }\ln r_{n-1}:=\left( 
\begin{array}{c}
\ln r_{n-1}^{(a)} \\ 
\ln r_{n-1}^{(b)}%
\end{array}%
\right) \text{,}  \label{3}
\end{equation}%
and $K\in \mathcal{M}_{2}$. Denote by $v^{[i]}$ the $i$-th element of a
vector $v$, and $R\in \mathcal{M}_{2}$ is defined by $\left( 
\begin{array}{cc}
1 & 1 \\ 
1 & 1%
\end{array}%
\right) =R\left( 
\begin{array}{cc}
2 & 0 \\ 
0 & 0%
\end{array}%
\right) R^{-1}$. Note that $\left( \ln r_{n}\right) ^{[1]}=\ln r_{n-1}^{(a)}$
and $\left( \ln r_{n}\right) ^{[2]}=\ln r_{n-1}^{(b)}$.

\begin{proposition}
\label{Prop: 3}Let $F$ be an SNRE which admits a dominate symbol $a\in 
\mathcal{A}\cap \mathcal{E}(F)$, then $F_{1}^{(a)}\neq b_{n-1}^{2}$.
Furthermore, we have
\end{proposition}

\begin{enumerate}
\item if $F_{1}^{(a)}=a_{n-1}^{2}$, then 
\begin{equation*}
h(X^{\mathcal{B}})=\frac{1}{4}\lim_{n\rightarrow \infty }\left( \ln
a_{1}+\sum_{j=1}^{n-1}2^{-j}\ln r_{j}^{(a)}\right) ;
\end{equation*}

\item if $F_{1}^{(a)}=a_{n-1}b_{n-1}$ and $F_{1}^{(b)}=a_{n-1}b_{n-1}$, then 
\begin{equation*}
h(X^{\mathcal{B}})=\frac{1}{4}\lim_{n\rightarrow \infty }\left( (\widehat{v}%
_{1})^{[1]}+\sum_{j=1}^{n-1}2^{-j}\left( \ln \widehat{r}_{j}\right)
^{[1]}\right) ,
\end{equation*}
where $\widehat{v}_{1}=R^{-1}v_{1}$ and $\ln \widehat{r}_{j}=R^{-1}\ln r_{j}$%
;

\item if $F_{1}^{(a)}=a_{n-1}b_{n-1}$ and $F_{1}^{(b)}=b_{n-1}^{2}$, then
the number $\left\vert B_{n}(X^{\mathcal{B}})\right\vert $ is up to
exponential. In this case, we have $h(X^{\mathcal{B}})=0$.
\end{enumerate}

\begin{proof}
Since $a$ is a dominate symbol, i.e., $a_{n}\geq b_{n}$, the equality $%
F_{1}^{(a)}=b_{n-1}^{2}$ implies that $F_{1}^{(b)}=b_{n-1}^{2}$ and so $%
b_{n}=1$ for all $n$, i.e., $b\in \mathcal{I}(F)$. Consequently, $%
F_{1}^{(a)}=b_{n-1}^{2}=1$; that is, $a\in \mathcal{I}(F)$, which is a
contradiction. Hence $F_{1}^{(a)}\neq b_{n-1}^{2}$. Thus, there are two
possibilities for $F_{1}^{(a)}$; namely, $F_{1}^{(a)}=a_{n-1}^{2}$ or $%
F_{1}^{(a)}=a_{n-1}b_{n-1}$.

\textbf{(i)} $F_{1}^{(a)}=F_{1}^{(b)}=a_{n-1}^{2}$. It follows the preceding
algorithm we have $K=\left( 
\begin{array}{cc}
2 & 0 \\ 
2 & 0%
\end{array}%
\right) $. Iterates the recursive formula (\ref{3}) we obtain $%
v_{n}=K^{n-1}v_{1}+\sum_{j=1}^{n-1}K^{n-j-1}\ln r_{j}$. Combining this with
the fact that $K^{n}=\left( 
\begin{array}{cc}
2^{n} & 0 \\ 
2^{n} & 0%
\end{array}%
\right) $ we have 
\begin{eqnarray}
\alpha _{n} &=&\ln a_{n}=\left( K^{n-1}v_{1}\right)
^{[1]}+\sum_{j=1}^{n-1}\left( K^{n-j-1}\ln r_{j}\right) ^{[1]}  \notag \\
&=&2^{n-1}\ln a_{1}+\sum_{j=1}^{n-1}2^{n-j-1}\ln r_{j}^{(a)}  \notag \\
&=&2^{n-1}\left( \ln a_{1}+\sum_{j=1}^{n-1}2^{-j}\ln r_{j}^{(a)}\right) 
\text{.}  \label{16}
\end{eqnarray}%
Let $A_{n}=\ln a_{1}+\sum_{j=1}^{n-1}2^{-j}\ln r_{j}^{(a)}$. Since $1\leq
r_{j}^{(a)}\leq 4$ for all $j$ (see (\ref{23})), $A_{n}$ converges and
denotes $A_{\infty }:=\lim_{n\rightarrow \infty }A_{n}$. It follows from
Lemma \ref{Lma: 2} that%
\begin{equation*}
h(X^{\mathcal{B}})=\lim_{n\rightarrow \infty }\frac{\alpha _{n}}{2^{n+1}-1}=%
\frac{1}{4}\left( \ln a_{1}+\sum_{j=1}^{\infty }2^{-j}\ln r_{j}^{(a)}\right)
=\frac{A_{\infty }}{4}\text{.}
\end{equation*}

\textbf{(ii)} $F_{1}^{(a)}=a_{n-1}^{2}$, $F_{1}^{(b)}=a_{n-1}b_{n-1}$ or $%
F_{1}^{(b)}=b_{n-1}^{2}$. If $F_{1}^{(b)}=a_{n-1}b_{n-1}$, we see that $%
K=\left( 
\begin{array}{cc}
2 & 0 \\ 
1 & 1%
\end{array}%
\right) $. Since $K^{n}=\left( 
\begin{array}{cc}
2^{n} & 0 \\ 
2^{n}-1 & 1%
\end{array}%
\right) $, identical argument is applied to show that $h(X^{\mathcal{B}})=%
\frac{A_{\infty }}{4}$. Similarly, in the case where $%
F_{1}^{(b)}=b_{n-1}^{2} $, we have $h(X^{\mathcal{B}})=\frac{A_{\infty }}{4}$%
.

\textbf{(iii) }$F_{1}^{(a)}=a_{n-1}b_{n-1}$ and $F_{1}^{(b)}=a_{n-1}b_{n-1}$%
. We see that $K=\left( 
\begin{array}{cc}
1 & 1 \\ 
1 & 1%
\end{array}%
\right) $. Let $R\in \mathcal{M}_{2}$ be defined as above. By a similar
argument, we have $v_{n}=R\left(
D^{n-1}R^{-1}v_{1}+\sum_{j=1}^{n-1}D^{n-j-1}R^{-1}\ln r_{j}\right) $. Let $%
R^{-1}v_{1}=\widehat{v}_{1}$ and $R^{-1}\ln r_{j}=\ln \widehat{r}_{j}$. Then
we have $v_{n}=R(D^{n-1}\widehat{v}_{1}+\sum_{j=1}^{n-1}D^{n-j-1}\ln 
\widehat{r}_{j})$ and 
\begin{eqnarray*}
\ln a_{n} &=&v_{n}^{[1]}=R_{11}\left( 2^{n-1}\widehat{v}_{1}{}^{[1]}+%
\sum_{j=1}^{n-1}2^{n-j-1}\left( \ln \widehat{r}_{j}\right) ^{[1]}\right) \\
&=&2^{n-1}\left[ R_{11}\left( \widehat{v}_{1}{}^{[1]}+\sum_{j=1}^{n-1}2^{-j}%
\left( \ln \widehat{r}_{j}\right) ^{[1]}\right) \right] \text{.}
\end{eqnarray*}%
Since $R_{11}=1$, $h(X^{\mathcal{B}})=\frac{1}{4}\lim_{n\rightarrow \infty
}\left( \widehat{v}_{1}{}^{[1]}+\sum_{j=1}^{n-1}2^{-j}\left( \ln \widehat{r}%
_{j}\right) ^{[1]}\right) $.

\textbf{(iv)} $F_{1}^{(a)}=a_{n-1}b_{n-1}$ and $F_{1}^{(b)}=b_{n-1}^{2}$. In
this case one can check that $b\in \mathcal{I}(F)$, $K=\left( 
\begin{array}{cc}
1 & 1 \\ 
0 & 2%
\end{array}%
\right) $ and $K^{n}=\left( 
\begin{array}{cc}
1 & 2^{n}-1 \\ 
0 & 2^{n}%
\end{array}%
\right) $. Since $a\in \mathcal{E}(F)$ and $\left( \ln r_{j}\right) ^{[2]}=0$
($b_{n}=1$), we have 
\begin{eqnarray*}
\alpha _{n} &=&v_{n}^{[1]}=\left( K^{n-1}v_{1}\right)
^{[1]}+\sum_{j=1}^{n-1}\left( K^{n-j-1}\ln r_{j}\right) ^{[1]} \\
&=&\ln a_{1}+\sum_{j=1}^{n-1}\left( \ln r_{j}\right) ^{[1]}=\ln
a_{1}+\sum_{j=1}^{n-1}\ln r_{n}^{(a)}\leq cn
\end{eqnarray*}%
for some constant $c$. Thus, the growth rate of $a_{n}$ is up to
exponential. Meanwhile, $h(X^{\mathcal{B}})=0$. This completes the proof.
\end{proof}

\begin{example}
Let $F=\{F^{(a)},F^{(b)}\}$ is defined as 
\begin{equation*}
\left\{ 
\begin{array}{c}
a_{n}=2a_{n-1}b_{n-1}, \\ 
b_{n}=b_{n-1}^{2}, \\ 
a_{1}=2,b_{1}=1.%
\end{array}%
\right.
\end{equation*}%
We have $v_{n}=Kv_{n-1}+\ln r_{n-1}$, where $K=\left( 
\begin{array}{cc}
1 & 1 \\ 
0 & 2%
\end{array}%
\right) $ and $\ln r_{n}=\left( 
\begin{array}{c}
\ln 2 \\ 
0%
\end{array}%
\right) $ for $n\in \mathbb{N}$. Thus $v_{n}=K^{n-1}v_{1}+K^{n-2}\ln
r_{1}+\cdots +\ln r_{n-1}$. Since $K^{n}=\left( 
\begin{array}{cc}
1 & 2^{n}-1 \\ 
0 & 2^{n}%
\end{array}%
\right) $, we have $\alpha _{n}=\alpha _{1}+\ln 2+\cdots +\ln 2=n\ln 2$.
Then, $(a_{n},b_{n})=(2^{n},1)$ for $n\geq 1$ due to the fact that $b\in 
\mathcal{I}(F)$.
\end{example}

\subsection{Cooperating type}

Let $F$ be an SNRE, and define $c_{n}=a_{n}+b_{n}$. We say that $F$ is of
the \emph{cooperating type }(type $\mathbf{C}$) if $%
c_{n}=c_{n-1}^{2}+g_{n-1} $, where $g_{n}\leq c_{n}^{2}$. For example, if $F$
is 
\begin{equation}
\left\{ 
\begin{array}{c}
a_{n}=2a_{n-1}b_{n-1}+b_{n-1}^{2}, \\ 
b_{n}=a_{n-1}^{2}+a_{n-1}b_{n-1}, \\ 
a_{1}=3,b_{1}=2.%
\end{array}%
\right.  \label{13}
\end{equation}%
Then we have $(a_{n}+b_{n})=(a_{n-1}+b_{n-1})^{2}+a_{n-1}b_{n-1}$ and $%
a_{n-1}b_{n-1}\leq (a_{n-1}+b_{n-1})^{2}$.

\begin{proposition}
\label{Prop: 4}Let $X^{\mathcal{B}}$ be an $S_{d}$-SFT and $F$ be its SNRE.
If $F$ is of the cooperating type, then 
\begin{equation*}
h(X^{\mathcal{B}})=\frac{1}{4}\left[ \ln c_{1}+\lim_{n\rightarrow \infty
}\sum_{j=1}^{n-1}2^{n-j-1}\ln (1+\frac{g_{j}}{c_{j}^{2}})\right] \text{.}
\end{equation*}
\end{proposition}

\begin{proof}
The proof is almost identical to the proof of Proposition \ref{Prop: 3},
thus we omit it.
\end{proof}

\begin{example}
Let $F$ be defined in (\ref{13}). We have $c_{n}=c_{n-1}^{2}+g_{n-1},$ where 
$g_{n}=a_{n}b_{n}$. Since $\frac{g_{n}}{c_{n}^{2}}\leq 1$, we have 
\begin{equation*}
\theta _{n}=2^{n-1}(\theta _{1}+\sum_{j=1}^{n-1}2^{n-j-1}\ln (1+\frac{g_{j}}{%
c_{j}^{2}})),
\end{equation*}%
where $\theta _{n}=\ln c_{n}$. Then $h(X^{\mathcal{B}})=\frac{1}{4}A_{\infty
}$, where 
\begin{equation*}
A_{\infty }=\lim_{n\rightarrow \infty }(\theta
_{1}+\sum_{j=1}^{n-1}2^{n-j-1}\ln (1+\frac{g_{j}}{c_{j}^{2}}))\text{.}
\end{equation*}
\end{example}

\subsection{Oscillating type}

We say that $F$ is of the $\emph{oscillating}$ $\emph{type}$ (type $\mathbf{O%
}$) if there exist two subsequences of $\mathbb{N}$, say $\{n_{m}^{(a)}\}$
and $\{n_{m}^{(b)}\}$ such that $a_{n}\geq b_{n}$ on $\{n_{m}^{(a)}\}$ and $%
a_{n}<b_{n}$ on $\{n_{m}^{(b)}\}$.

\begin{proposition}
\label{Prop: 5}Let $\mathcal{B}=\{(1,1,2),(1,2,2),(2,1,1)\}$ be a basic set.
Suppose $X^{\mathcal{B}}$ and $F$ are the corresponding $S_{2}$-SFT and SNRE
respectively. Then $F$ is of the oscillating type, and 
\begin{equation*}
h(X^{\mathcal{B}})=\frac{1}{4}\lim_{n\rightarrow \infty }\left( \alpha
_{1}+\sum_{j=1}^{n}2^{-2j}\ln r_{2j-1}^{(a)}\right) \text{.}
\end{equation*}
\end{proposition}

\begin{proof}
Note that the SNRE $F$ is of the following form 
\begin{equation}
\left\{ 
\begin{array}{c}
a_{n}=a_{n-1}b_{n-1}+b_{n-1}^{2}, \\ 
b_{n}=a_{n-1}^{2}, \\ 
a_{1}=2,\text{ }b_{1}=1.%
\end{array}%
\right.  \label{1}
\end{equation}%
Let $\tau _{n}=\frac{a_{n}}{b_{n}}$, it can be easily checked that $\tau
_{n}=\frac{1}{\tau _{n-1}}+\frac{1}{\tau _{n-1}^{2}}$. Note that $\tau
_{1}=2 $ and $\tau _{2}=\frac{3}{4}$. We claim that $\tau _{n}\geq 2$
implies $\tau _{n+1}\leq \frac{3}{4}$ and $\tau _{n}\leq \frac{3}{4}$ infers 
$\tau _{n+1}\geq 2$. Indeed, if $\tau _{n}\geq 2$ then $\tau _{n+1}=\frac{1}{%
\tau _{n}}+\left( \frac{1}{\tau _{n}}\right) ^{2}\leq \frac{1}{2}+\left( 
\frac{1}{2}\right) ^{2}=\frac{3}{4}$. Whenever $\tau _{n}\leq \frac{3}{4}$,
then $\tau _{n+1}\geq \frac{4}{3}+(\frac{4}{3})^{2}=\frac{28}{9}\geq 2$.
Thus we have $\gamma _{2n+1}>1$ and $\gamma _{2n}<1$. This means $a_{n}>b_{n}
$ if $n$ is odd and $a_{n}<b_{n}$ if $n$ is even. That is, $F$ is of the
oscillating type. In this case, there is no dominate symbol. However, if $n$
is odd, $a_{n}$ is still a dominate symbol, and $b_{n}$ is a dominate symbol
for $n$ being even. Expand $F$ to the $(n-2)$ order and arrange $F^{(a)}$
and $F^{(b)}$ as follows.%
\begin{equation}
\left\{ 
\begin{array}{c}
a_{n+2}=a_{n}^{4}+a_{n}^{3}b_{n}+a_{n}^{2}b_{n}^{2}, \\ 
b_{n+2}=a_{n}^{2}b_{n}^{2}+2a_{n}b_{n}^{3}+b_{n}^{4}.%
\end{array}%
\right.  \label{2}
\end{equation}%
Recall that $\alpha _{n}=\ln a_{n}$, $\beta _{n}=\ln b_{n}$, and $%
v_{n}=\left( 
\begin{array}{c}
\alpha _{n} \\ 
\beta _{n}%
\end{array}%
\right) $. We have $v_{n+2}=Pv_{n}+\ln r_{n}$, where $P=\left( 
\begin{array}{cc}
2^{2} & 0 \\ 
2 & 2%
\end{array}%
\right) $, and 
\begin{equation*}
\ln r_{n}=\left( 
\begin{array}{c}
\ln (1+\frac{b_{n}}{a_{n}}+\left( \frac{b_{n}}{a_{n}}\right) ^{2}) \\ 
\ln (1+2\frac{b_{n}}{a_{n}}+\left( \frac{b_{n}}{a_{n}}\right) ^{2})%
\end{array}%
\right) =\left( 
\begin{array}{c}
\ln r_{n}^{(a)} \\ 
\ln r_{n}^{(b)}%
\end{array}%
\right) \text{.}
\end{equation*}%
Iterating $v_{n+2}$ we obtain that $v_{2n+1}=P^{n}\alpha
_{1}+\sum_{j=1}^{n}P^{n-j}\ln r_{2j-1}$. Since $P^{n}=\left( 
\begin{array}{cc}
2^{2n} & 0 \\ 
2^{2n}-2^{n} & 2^{n}%
\end{array}%
\right) $, we have%
\begin{equation}
\alpha _{2n+1}=2^{2n}\left( \alpha _{1}+\sum_{j=1}^{n}2^{-2j}\ln
r_{2j-1}^{(a)}\right) \text{.}  \label{25}
\end{equation}%
Set $A_{n}=\left( \alpha _{1}+\sum_{j=1}^{n}2^{-2j}\ln r_{2j-1}^{(a)}\right) 
$, then $\lim_{n\rightarrow \infty }A_{n}$ exists and define $A_{\infty
}=\lim_{n\rightarrow \infty }A_{n}$. Note that $\left\vert
E_{2n+1}\right\vert =\sum_{j=0}^{2n+1}2^{j}=2^{2n+2}-1$. Since $a$ is a
dominate symbol for $n$ is odd. Combining Lemma \ref{Lma: 2}, Proposition %
\ref{Prop: 3} with (\ref{25}) yields 
\begin{equation}
\lim_{n\rightarrow \infty }\frac{\ln a_{2n+1}}{\left\vert
E_{2n+1}\right\vert }=\frac{1}{4}\lim_{n\rightarrow \infty }\left( \alpha
_{1}+\sum_{j=1}^{n}2^{-2j}\ln r_{2j-1}^{(a)}\right) =\frac{A_{\infty }}{4}%
\text{,}  \label{14}
\end{equation}%
Combining (\ref{14}) with the fact that the limit (\ref{17}) exists, we have 
$h(X^{\mathcal{B}})=\lim_{n\rightarrow \infty }\frac{\ln a_{2n+1}}{%
\left\vert E_{2n+1}\right\vert }=\frac{A_{\infty }^{(a)}}{4}$. This
completes the proof.
\end{proof}

\subsection{Complete characterization}

Recall that $\left\vert F^{(\ast )}\right\vert $ denotes the number of items
of $F^{(\ast )}$ for $\ast =a$ or $b$.

\begin{lemma}
\label{Lma: 4}Let $F$ be such that 
\begin{equation*}
\left\{ 
\begin{array}{c}
a_{n}=2a_{n-1}b_{n-1}, \\ 
b_{n}=a_{n-1}^{2}, \\ 
a_{1}=2\text{, }b_{1}=1.%
\end{array}%
\right.
\end{equation*}%
Then $F$ is of the dominating type.
\end{lemma}

\begin{proof}
Let $\tau _{n}=\frac{a_{n}}{b_{n}}$. Then we have $\tau _{n}=2\frac{b_{n-1}}{%
a_{n-1}}=2\frac{1}{\tau _{n-1}}$, i.e., $\tau _{n}\tau _{n-1}=2$. Since $%
\tau _{1}=2$ and $\tau _{2}=1$, it can be easily checked that $\tau
_{2n-1}=2 $ and $\tau _{2n}=1$. This shows that $F$ is of the dominating
type.
\end{proof}

\begin{theorem}
\label{Thm: 3}Let $(d,k)=(2,2)$ and $X^{\mathcal{B}}$ be an $S_{d}$-SFT.
Suppose $F$ is its SNRE, then $F$ is either one of the following four types.

\begin{enumerate}
\item $F$ is of the equal growth type,

\item $F$ is of the dominating type,

\item $F$ is of the oscillating type,

\item $F$ is of the cooperating type.
\end{enumerate}
\end{theorem}

\begin{proof}
Recall that $\left\vert F^{(\ast )}\right\vert $ is the number of items in $%
F^{(\ast )}$. Without loss of generality, we only discuss the case where $%
\left\vert F^{(a)}\right\vert >\left\vert F^{(b)}\right\vert $, the case
where $\left\vert F^{(a)}\right\vert <\left\vert F^{(b)}\right\vert $ is
similar. For the case where $\left\vert F^{(a)}\right\vert =\left\vert
F^{(b)}\right\vert $, Proposition \ref{Prop: 2} is applied to show it is of
type $\mathbf{E}$. Thus we divide the discussion into the following small
cases. \textbf{(i)} $\left\vert F^{(a)}\right\vert =4$. In this case, we
know that $a_{n}\geq b_{n}$, thus $a$ is a dominate symbol. Thus $F\in 
\mathbf{D}$; \textbf{(ii)} $\left\vert F^{(a)}\right\vert =3$. We only
discuss the following subcases (a) $\left\vert F^{(b)}\right\vert =1$. Since 
$\left\vert F^{(a)}\right\vert =3$ there are only three possibilities: (1) $%
a_{n}=a_{n-1}^{2}+2a_{n-1}b_{n-1}$; (2) $a_{n}=2a_{n-1}b_{n-1}+b_{n-1}^{2}$;
(3) $a_{n}=a_{n-1}^{2}+a_{n-1}b_{n-1}+b_{n-1}^{2}$. In case (3), since $%
\left\vert F^{(b)}\right\vert =1$, we conclude that $a$ is a dominate
symbol. Thus $F\in \mathbf{D}$. In case (2), if $b_{n}=a_{n-1}b_{n-1}$ or $%
b_{n-1}^{2}$, then $a$ is also a dominate symbol and $F\in \mathbf{D}$. Thus
we only to discuss $b_{n}=a_{n-1}^{2}$. That is 
\begin{equation*}
F_{I}=\left\{ 
\begin{array}{c}
a_{n}=2a_{n-1}b_{n-1}+b_{n-1}^{2}, \\ 
b_{n}=a_{n-1}^{2}, \\ 
a_{1}=3,b_{1}=1.%
\end{array}%
\right.
\end{equation*}%
It can be easily checked that $F_{I}\in \mathbf{C}$. For case (1), if $%
b_{n}=a_{n-1}^{2}$ or $a_{n-1}b_{n-1}$, then $a$ is a dominate symbol and $%
F\in \mathbf{D}$, thus we only need to discuss $b_{n}=b_{n-1}^{2}$. However,
in this case $b_{n}=1$, i.e., $b\in \mathcal{I}(F)$. Thus $a$ is still a
dominate symbol, i.e., $F\in \mathbf{D}$. (b) $\left\vert F^{(b)}\right\vert
=2$. It follows the same argument as case (a), we also has three cases (1) -
(3) and $b_{n}$ has the following cases: $b_{n}=a_{n-1}^{2}+a_{n-1}b_{n-1}$
or $2a_{n-1}b_{n-1}$ or $a_{n-1}b_{n-1}+b_{n-1}^{2}$ or $%
a_{n-1}^{2}+b_{n-1}^{2}$. Under the same arguments as above, we only need to
discuss the following cases.%
\begin{eqnarray*}
F_{II} &=&\left\{ 
\begin{array}{c}
a_{n}=2a_{n-1}b_{n-1}+b_{n-1}^{2}, \\ 
b_{n}=a_{n-1}^{2}+a_{n-1}b_{n-1}, \\ 
a_{1}=3,b_{1}=2.%
\end{array}%
\right. \text{ } \\
F_{III} &=&\left\{ 
\begin{array}{c}
a_{n}=a_{n-1}^{2}+a_{n-1}b_{n-1}+b_{n-1}^{2}, \\ 
b_{n}=2a_{n-1}b_{n-1}, \\ 
a_{1}=3,b_{1}=2.%
\end{array}%
\right. \\
F_{IV} &=&\left\{ 
\begin{array}{c}
a_{n}=a_{n-1}^{2}+2a_{n-1}b_{n-1}, \\ 
b_{n}=a_{n-1}b_{n-1}+b_{n-1}^{2}, \\ 
a_{1}=3,b_{1}=2.%
\end{array}%
\right.
\end{eqnarray*}%
For $F_{II}$, we see that $%
(a_{n}+b_{n})=(a_{n-1}+b_{n-1})^{2}+a_{n-1}b_{n-1} $, thus $F\in \mathbf{C}$%
. For $F_{III}$, let $\tau _{n}=\frac{a_{n}}{b_{n}}$, then we have $\tau
_{n}=\frac{1}{2}+\frac{\tau _{n-1}}{2}+\frac{1}{2\tau _{n-1}}$. Since $\tau
_{1}=\frac{3}{2}>1$, we conclude that $\tau _{n}>1$ for all $n\geq 1$ by
induction, i.e., $a$ is a dominate symbol and $F_{III}\in \mathbf{D}$. The
discussion of $F_{IV}$ is the same as $F_{II}$; that is, $F_{IV}\in \mathbf{C%
}$; \textbf{(iii)} $\left\vert F^{(a)}\right\vert =2$. There are only four
possibilities: (1) $a_{n}=a_{n-1}^{2}+a_{n-1}b_{n-1}$; (2) $%
a_{n}=2a_{n-1}b_{n-1}$; (3) $a_{n}=a_{n-1}b_{n-1}+b_{n-1}^{2}$ and (4). $%
a_{n}=a_{n-1}^{2}+b_{n-1}^{2}$. Since $\left\vert F^{(a)}\right\vert
>\left\vert F^{(b)}\right\vert $, we only need to discuss the case where $%
\left\vert F^{(b)}\right\vert =1$. That is, $b_{n}=a_{n-1}^{2}$ or $%
a_{n-1}b_{n-1}$ or $b_{n-1}^{2}$. If $b_{n}=b_{n-1}^{2}$, then $b\in 
\mathcal{I}(F)$ and $a$ is a dominate symbol; this means $F\in \mathbf{D}$.
Thus we only have the following three cases.%
\begin{eqnarray*}
F_{V} &=&\left\{ 
\begin{array}{c}
a_{n}=2a_{n-1}b_{n-1}, \\ 
b_{n}=a_{n-1}^{2}, \\ 
a_{1}=2,b_{1}=1.%
\end{array}%
\right. \text{ }F_{VI}=\left\{ 
\begin{array}{c}
a_{n}=a_{n-1}b_{n-1}+b_{n-1}^{2}, \\ 
b_{n}=a_{n-1}^{2}, \\ 
a_{1}=2,b_{1}=1.%
\end{array}%
\right. \\
F_{VII} &=&\left\{ 
\begin{array}{c}
a_{n}=a_{n-1}^{2}+b_{n-1}^{2}, \\ 
b_{n}=a_{n-1}b_{n-1}, \\ 
a_{1}=2,b_{1}=1.%
\end{array}%
\right.
\end{eqnarray*}%
For case $F_{V}$, Lemma \ref{Lma: 4} shows that it is of the type $\mathbf{D}
$. The proof of Proposition \ref{Prop: 5} indicates that $F_{VI}\in \mathbf{O%
}$. For case $F_{VII}$, we assume $\tau _{n}=\frac{a_{n}}{b_{n}}=\tau _{n-1}+%
\frac{1}{\tau _{n-1}}$. Since $\tau _{1}\geq 2$, it can be checked that $%
\tau _{n}\geq 2$ for all $n$ by induction. Thus $a$ is a dominate symbol and 
$F_{VII}\in \mathbf{D}$. This completes the proof.
\end{proof}

\subsection{Numerical results}

Recall that $v_{F}=(f_{\mathbf{a}})_{\mathbf{a}\in \mathcal{A}^{d}}$ is the
indicator vector of $F$ (Definition \ref{Def: 1}). For $k=2$ we denote by $%
v^{(a)}=v_{F^{(a)}}$ and $v^{(b)}=v_{F^{(b)}}$. We provide the numerical
results for the entropies of the case where $(d,k)=(2,2)$. Due to the
symmetry of the SNRE, i.e., the cases of $(v^{(a)},v^{(b)})=(v_{1},v_{2})$
and $(v^{(a)},v^{(b)})=(v_{2},v_{1})$ define the same SNRE for vectors $%
v_{1} $ and $v_{2}$. It reduces the discussion to those cases of $\left\vert
F^{(a)}\right\vert >\left\vert F^{(b)}\right\vert $. The case where $%
\left\vert F^{(a)}\right\vert =\left\vert F^{(b)}\right\vert $ immediately
follows from the Proposition \ref{Prop: 2}, thus we omit it. The following
tables are the entropies of the cases where $\left\vert F^{(a)}\right\vert
=2,3$, and $4$, respectively.

\begin{equation*}
\begin{tabular}{|l|l|l|l|l|}
\hline
$v^{(b)}\backslash v^{(a)}$ & $(1,1,0)$ & $(1,0,1)$ & $(0,1,1)$ & $(0,2,0)$
\\ \hline
$(1,0,0)$ & 0.285443 & 0.254262 & 0.214332 & 0.346235 \\ \hline
$(0,1,0)$ & 0.253877 & 0.216424 & 0.252677 & 0.295580 \\ \hline
$(0,0,1)$ & 0.234348 & 0.203677 & 0 & 0 \\ \hline
\end{tabular}%
\end{equation*}%
\begin{equation*}
\begin{tabular}{|l|l|l|l|}
\hline
$v^{(b)}\backslash v^{(a)}$ & $(1,1,1)$ & $(1,2,0)$ & $(0,2,1)$ \\ \hline
$(1,0,0)$ & 0.404347 & 0.429271 & 0.517933 \\ \hline
$(0,1,0)$ & 0.346538 & 0.372742 & 0.427385 \\ \hline
$(0,0,1)$ & 0.325765 & 0.346574 & 0 \\ \hline
$(1,1,0)$ & 0.474630 & 0.490218 & 0.527259 \\ \hline
$(1,0,1)$ & 0.462992 & 0.480426 & 0.523983 \\ \hline
$(0,1,1)$ & 0.432619 & 0.451472 & 0.516799 \\ \hline
$(0,2,0)$ & 0.455134 & 0.472200 & 0.522268 \\ \hline
\end{tabular}%
\end{equation*}

\begin{equation*}
\begin{tabular}{|l|l|}
\hline
$v^{(b)}\backslash v^{(a)}$ & $(1,2,1)$ \\ \hline
$(1,0,0)$ & 0.508156 \\ \hline
$(0,1,0)$ & 0.432802 \\ \hline
$(0,0,1)$ & 0.407355 \\ \hline
$(1,1,0)$ & 0.570417 \\ \hline
$(1,0,1)$ & 0.556489 \\ \hline
$(0,1,1)$ & 0.507662 \\ \hline
$(0,2,0)$ & 0.537203 \\ \hline
$(1,1,1)$ & 0.625995 \\ \hline
$(1,2,0)$ & 0.633417 \\ \hline
$(0,2,1)$ & 0.611294 \\ \hline
\end{tabular}%
\end{equation*}

\section{Hom-shifts}

\subsection{Golden mean shifts and $k$-colored chessboard}

Let $k\geq 2$ and $\mathcal{A}$ be the set of symbols such that $|\mathcal{A}| =k$. Suppose $\mathbb{A}=(A_{i})_{i=1}^{d}\in \mathcal{M}_{k}^{d}$ is the $d$-tuple of the $k\times k$ binary matrices
such that $A_{i}$ is indexed by $\mathcal{A}$. The $S_{d}$\emph{-vertex shift%
} $X_{\mathbb{A}}$ is defined by $\mathbb{A}$ in which $A_{i}(a,b)=0$ if and
only if $\left( a,s_{i},b\right) \in \mathcal{F}$. That is 
\begin{equation*}
X_{\mathbb{A}}=\{x\in \mathcal{A}^{S_{d}}:A_{i}(x(g),x(gs_{i}))=1\text{ for }%
g\in G\text{, }1\leq i\leq d\}\text{.}
\end{equation*}%
In \cite{ban2017mixing}, we prove that an $S_{d}$-SFT is conjugate to an $%
S_{d}$-vertex shift and vice versa. We call an $S_{d}$-vertex shift $X=X_{%
\mathbb{A}}$ \emph{hom-shift }if $A_{i}=A\in \mathcal{M}_{k}$ for all $i$,
and we write $X_{\mathbb{A}}=X_{A}$ if it causes no confusion. $X_{\mathbb{A}%
}$ is called a \emph{golden mean shift} if $X$ is a hom-shift and $A$ is of
the following form%
\begin{equation}
A(i,j)=\left\{ 
\begin{array}{cc}
0\text{,} & i=j=k\text{;} \\ 
1\text{,} & \text{otherwise.}%
\end{array}%
\right.  \label{10}
\end{equation}%
A hom-shift $X_{A}$ is called \emph{k-colored chessboard }if 
\begin{equation*}
A(i,j)=\left\{ 
\begin{array}{cc}
0\text{,} & i=j\text{;} \\ 
1\text{,} & \text{otherwise.}%
\end{array}%
\right.
\end{equation*}

Recall that $\mathbf{D}$ (resp. $\mathbf{C,O,E}$) denotes the collection of
SNREs which belong to the dominating type (resp. cooperating, oscillating
and equal growth types)

\begin{theorem}
\label{Thm: 1}Let $d,k\in \mathbb{N}$.

\begin{enumerate}
\item If $X_{A}$ is a golden mean shift, then $X_{A}\in \mathbf{C}\cap 
\mathbf{D}$\textbf{.}

\item If $X_{A}$ is a $k$-colored chessboard, then $X_{A}\in \mathbf{E}$%
\textbf{.}
\end{enumerate}
\end{theorem}

\begin{proof}
Suppose $X_{A}$ is a golden mean shift. For clarity, we only prove the case
where $k=2$, and the other cases can be treated similarly. Let $\mathcal{A}%
=\{1,2\}$. Since $A$ is defined in (\ref{10}), we have $\mathcal{B}%
^{(1)}=\{(1,i_{1},i_{2},\ldots ,i_{d}):i_{j}=1$ or $2$ for $1\leq j\leq d\}$
and $\mathcal{B}^{(2)}=\{(2,1,1,\ldots ,1)\}$. Thus, the corresponding SNRE
with respect to $\mathcal{B}=\mathcal{B}^{(1)}\cup \mathcal{B}^{(2)}$ is of
the following form.%
\begin{equation}
\left\{ 
\begin{array}{c}
a_{n}=a_{n-1}^{d}+C_{1}^{d}a_{n-1}^{d-1}b_{n-1}+C_{2}^{d}a_{n-1}^{d-2}b_{n-1}^{2}+\cdots +b_{n-1}^{d}%
\text{,} \\ 
b_{n-1}=a_{n-1}^{d}\text{,} \\ 
a_{1}=2^{d}\text{ and }b_{1}=1\text{. }%
\end{array}%
\right.  \label{15}
\end{equation}%
It can be easily checked that $F=\{F^{(a)},F^{(n)}\}\in \mathbf{C}\cap 
\mathbf{D}$ under the same discussion in Proposition \ref{Prop: 3} and
Proposition \ref{Prop: 4}. Suppose $X_{A}$ is a $k$-colored chessboard. We
have $\mathcal{B}^{(i)}=\{(i,i_{1},i_{2},\ldots ,i_{d}):i_{j}\neq i$ for $%
1\leq j\leq d\}$. The corresponding SNRE is of the following form $%
a_{n}^{(i)}=(k-1)^{d}a_{n-1}^{(1)}\cdots
a_{n-1}^{(i-1)}a_{n-1}^{(i+1)}\cdots a_{n-1}^{(k)}$. Thus, $F\in \mathbf{E}$%
. The proof is complete.
\end{proof}

\begin{corollary}
\label{Cor: 1}Let $2\leq d\in \mathbb{N}$. Then

\begin{enumerate}
\item The entropy of an $S_{d}$-GMS for $k=2$ is $h(X_{A})=\frac{A_{\infty }%
}{d^{2}}$, where 
\begin{equation*}
A_{\infty }=\lim_{n\rightarrow \infty }A_{n}=d\ln 2+\lim_{n\rightarrow
\infty }\sum_{j=1}^{n-1}d^{-j}\ln r_{j}^{(a)}\text{.}
\end{equation*}

\item The entropy of a $k$-colored chessboard is $d\ln (k-1)$.
\end{enumerate}
\end{corollary}

\begin{proof}
It follows from Theorem \ref{Thm: 1} that $X_{A}\in \mathbf{D}$ and $a$ is
the dominate symbol. Under the identical argument as the proof of
Proposition \ref{Prop: 3}, we have $h(X_{A})=\frac{A_{\infty }}{d^{2}}$,
where $A_{\infty }=\lim_{n\rightarrow \infty }A_{n}=d\ln
2+\lim_{n\rightarrow \infty }\sum_{j=1}^{n-1}d^{-j}\ln r_{j}^{(a)}$. If $%
X_{A}$ is a $k$-colored chessboard, its SNRE is of type $\mathbf{E}$. It
follows from Proposition \ref{Prop: 2} that $h(X_{A})=d\ln (k-1)$. This
completes the proof.
\end{proof}

The method developed in $S_{2}$-GMS can also be applied to the traditional
one-dimensional GMS.

\begin{example}
Let $X_{A}$ be the golden mean shift of $(d,k)=(1,2)$. That is, $%
X_{A}=\{x\in \{1,2\}^{\mathbb{N}\cup \{0\}}:A(x_{i},x_{i+1})=1$ for $i\geq
1\}$, where $A=\left( 
\begin{array}{cc}
1 & 1 \\ 
1 & 0%
\end{array}%
\right) $. The SNRE (\ref{15}) gives 
\begin{equation}
\left\{ 
\begin{array}{c}
a_{n}=a_{n-1}+b_{n-1}, \\ 
b_{n}=a_{n-1}, \\ 
a_{1}=2\text{ and }b_{1}=1.\text{ }%
\end{array}%
\right.  \label{27}
\end{equation}%
Note that (\ref{27}) is an linear recursive equation, it follows from (\ref%
{16}) we have 
\begin{eqnarray*}
\ln a_{n} &=&A_{n}=\ln 2+\sum_{j=1}^{n-1}\ln r_{j}^{(a)}=\ln 2+\ln
\prod_{j=1}^{n-1}r_{j}^{(a)} \\
&=&\ln \prod_{j=1}^{\infty }2\left( 1+\frac{b_{j}}{a_{j}}\right) =\ln
2\left( 1+\frac{1}{2}\right) \left( 1+\frac{2}{3}\right) \cdots \left( 1+%
\frac{b_{n-1}}{a_{n-1}}\right) \text{.}
\end{eqnarray*}%
In this case $\left\vert E_{n}\right\vert =n,$ thus 
\begin{eqnarray*}
h(X_{A}) &=&\lim_{n\rightarrow \infty }\frac{A_{n}}{n}=\lim_{n\rightarrow
\infty }\frac{\ln a_{n}}{n} \\
&=&\lim_{n\rightarrow \infty }\frac{1}{n}\ln 2\left( 1+\frac{1}{2}\right)
\left( 1+\frac{2}{3}\right) \cdots \left( 1+\frac{b_{n-1}}{a_{n-1}}\right) \\
&=&\lim_{n\rightarrow \infty }\frac{1}{n}\ln \left( a_{n-1}+b_{n-1}\right)
=\lim_{n\rightarrow \infty }\frac{1}{n}\ln a_{n}=\ln g\text{,}
\end{eqnarray*}%
where $a_{n}$ is the $n$-th Fibonacci number and $g=\frac{1+\sqrt{5}}{2}$.
\end{example}

\subsection{Golden mean shifts on free groups}

In this section we restrict our discussion to the entropy formula of $F_{d}$%
-GMSs. The previous results on $S_{d}$-SFTs can be applied to solve the
entropy formula of the $F_{d}$-GMS. Let $2\leq d,k\in \mathbb{N}$ and $%
q=2d-1 $.

\begin{theorem}
\label{Thm: 5}Let $X$ be the $F_{d}$-GMS, then we have $h(X)=\frac{A_{\infty
}}{q^{2}}$, where 
\begin{equation*}
A_{\infty }=q\ln k+\lim_{n\rightarrow \infty }\sum_{j=1}^{n-1}q^{-j}\ln
r_{j}^{(a)}\text{.}
\end{equation*}
\end{theorem}

\begin{proof}
For clarity, we focus on the case where $k=2$. Since $k=2$, we may assume that $\mathcal{A}%
=\{1,2\}$ and let $\widehat{a}_{n}$ (resp. $\widehat{b}_{n}$) be the number
of configurations of $X$ on $G_{n}=\{g\in F_{d}:\left\vert g\right\vert
\leq n\}$ with the root $\epsilon $ being colored by $1$ (resp. $2$). We
also denote by $a_{n}$ (resp. $b_{n}$) the number of configurations of $X$ on $E_{n}=\{g\in S_{q}:\left\vert g\right\vert \leq n\}$ with the root
being colored by $1$ (resp. $2$). Since $q=2d-1$, we have the following
formulae%
\begin{equation}
\left\{ 
\begin{array}{c}
\widehat{a}_{n}=a_{n-1}^{2d}+C_{1}^{2d}a_{n-1}^{2d-1}b_{n-1}+\cdots
+C_{1}^{2d}a_{n-1}b_{n-1}^{2d-1}+b_{n-1}^{2d}, \\ 
\widehat{b}_{n}=a_{n-1}^{2d},%
\end{array}%
\right.   \label{12}
\end{equation}%
and 
\begin{equation}
\left\{ 
\begin{array}{c}
a_{n}=a_{n-1}^{q}+C_{1}^{q}a_{n-1}^{q-1}b_{n-1}+\cdots
+C_{1}^{q}a_{n-1}b_{n-1}^{q-1}+b_{n-1}^{q}, \\ 
b_{n}=a_{n-1}^{q}.%
\end{array}%
\right.   \label{11}
\end{equation}%
Let $F=\{F^{(a)},F^{(b)}\}$ be the SNRE (\ref{11}), we know that $F\in 
\mathbf{D}$ (Theorem \ref{Thm: 1}). Then we have $\left\vert
E_{n}\right\vert =\sum_{i=0}^{n}q^{i}=\frac{q^{n+1}-1}{q-1}$ and $\left\vert
G_{n}\right\vert =1+2d\left\vert E_{n-1}\right\vert $. Since $a$ is a
dominate symbol, it follows from Lemma \ref{Lma: 2} and (\ref{12}) that 
\begin{eqnarray*}
h(X) &=&\lim_{n\rightarrow \infty }\frac{\ln \left( \widehat{a}_{n}+\widehat{%
b}_{n}\right) }{\left\vert G_{n}\right\vert }=\lim_{n\rightarrow \infty }%
\frac{\ln \widehat{a}_{n}}{\left\vert G_{n}\right\vert } \\
&=&\lim_{n\rightarrow \infty }\frac{\ln \left( a_{n-1}^{2d}+\cdots
+b_{n-1}^{2d}\right) }{1+2d\left\vert E_{n-1}\right\vert } \\
&=&\lim_{n\rightarrow \infty }\frac{2d\ln a_{n-1}}{\left\vert
E_{n-1}\right\vert }\left( \frac{\left\vert E_{n-1}\right\vert }{%
1+2d\left\vert E_{n-1}\right\vert }\right) =\lim_{n\rightarrow \infty }\frac{%
\ln a_{n-1}}{\left\vert E_{n-1}\right\vert }\text{.}
\end{eqnarray*}%
Proposition \ref{Prop: 3} is applied to show that $\lim_{n\rightarrow \infty
}\frac{\ln a_{n-1}}{\left\vert E_{n-1}\right\vert }=\frac{A_{\infty }}{q^{2}}
$, where 
\begin{equation*}
A_{\infty }=\lim_{n\rightarrow \infty }A_{n}=q\ln k+\lim_{n\rightarrow
\infty }\sum_{j=1}^{n-1}q^{-j}\ln r_{j}^{(a)}\text{.}
\end{equation*}%
Thus $h(X)=\frac{A_{\infty }}{q^{2}}$. This completes the proof.
\end{proof}

\section{Conclusion and open problems}

List results of this paper as follows.

\begin{enumerate}
\item Four types, namely $\mathbf{E}$, $\mathbf{D}$, $\mathbf{C}$, $\mathbf{O%
}$ types, are introduced. Their entropy formulae are presented in
Proposition \ref{Prop: 2}, Proposition \ref{Prop: 3}, Proposition \ref{Prop:
4}, and Proposition \ref{Prop: 5} respectively. Furthermore, the set of all
SNREs with $(d,k)=(2,2)$ is equal to $\mathbf{E}\cup \mathbf{D}\cup \mathbf{C%
}\cup \mathbf{O}$ (Theorem \ref{Thm: 3}). This gives a complete
characterization for the entropies of $S_{2}$-SFTs with two symbols.

\item Two types of hom-shifts on $S_{d}$, the $S_{d}$-GMS and $k$-colored
chessboard on $S_{d}$, are introduced. The entropy formulae of these two
types are presented (Theorem \ref{Thm: 1}). The entropy formula of the $%
F_{d} $-GMS is also developed by using the method of entropy theory on $S_{d}
$-SFTs (Theorem \ref{Thm: 5}).
\end{enumerate}

Although we give the characterization for $(d,k)=(2,2)$, the general entropy
formula for arbitrary $(d,k)\in \mathbb{N}\times \mathbb{N}$ is far from
being solved. We list some possible problems in the future study.

\begin{problem}[Priori criterion for an SNRE]
In Section 3, we introduce four types of SNRE on which we can compute their
entropies explicitly. But the following criteria seem more important for $%
\left\vert \mathcal{A}\right\vert \geq 3$.

\begin{enumerate}
\item How to check if $F$ admits a dominate symbols?

\item How to check if $F$ belongs to the oscillating type?

\item Can we characterize all SNREs completely for arbitrary $\left\vert 
\mathcal{A}\right\vert $?
\end{enumerate}
\end{problem}

\begin{problem}[zero order estimate]
Let $X^{\mathcal{B}}$ be an $F_{d}$-SFT or $S_{d}$-SFT. Recall that the
value $\kappa $ (resp. $\lambda _{2}$ and $\lambda _{1}$) is the $2^{nd}$%
-order (resp. $1^{st}$-order and zero-order) speed of $B_{n}(X^{\mathcal{B}%
}) $ (see (\ref{18})). We have shown so far that there exists an algorithm
to compute $\kappa $ \cite{ban2017tree} and $\lambda _{2}$. However, the
number $\lambda _{1}$ seems crucial for the accuracy of the number $%
\left\vert B_{n}(X^{\mathcal{B}})\right\vert $. In \cite%
{piantadosi2008symbolic}, the author shows that the value $\lambda _{1}$ may
not exist in the case of $F_{d}$-GMS. We may ask: under which conditions $%
\lambda _{1}$ exists, and how to compute it? Another question is whether there exists an intermediate order between $\lambda_1$ and $\lambda_2$.
\end{problem}

\begin{problem}[Entropy of SFTs on free groups ]
In Section 4, we study the entropy of $F_{d}$-GMS. However, the method can
not apply to general $F_{d}$-SFTs. One of the reasons is that the SNRE (\ref%
{12}) and (\ref{11}) depend on the length of $E_{n}$ of the groups $S_{q}$
(but the case of $F_{d}$-GMS will not). So, what is the entropy formula for
an $F_{d}$-SFT? Furthermore, if $G$ is a group which is not free, how to
compute the entropy of the $G$-SFT?
\end{problem}

% bibliography ---------------------------------------------------
\bibliographystyle{amsplain}
\bibliography{ban}

% -------------------------------------------------------------

\end{document}